\documentclass[12pt, a4paper]{article}
\usepackage[english]{babel}
\usepackage{amsmath,amssymb,amsfonts,amscd}
\usepackage{amsthm}
\usepackage[T1]{fontenc}
\usepackage{latexsym}
\usepackage{amssymb,amsbsy,amsmath,amscd}
\usepackage{amsmath,amssymb,amsbsy, amsmath, amsfonts, amscd}
\usepackage[mathcal]{eucal}
\usepackage[pdftex]{color,graphicx}
\usepackage{latexsym}
\usepackage{fancyhdr}
\usepackage{cite}
\newcommand{\cc}{\mathcal}
\newcommand{\rr}{\mathbb{R}}
\newcommand{\bb}{\mathbb}
\newcommand{\rn}{\mathbb{R}^N}
\newcommand{\sla}{\backslash}

\newcommand{\bdisp}{\begin{displaymath}}
\newcommand{\edisp}{\begin{displaymath}}
\newcommand{\bsplit}{\begin{split}}
\newcommand{\esplit}{\end{split}}
\newcommand{\bacc}{\left\{}
\newcommand{\eacc}{\right\}}
\newcommand{\bp}{\left(}
\newcommand{\ep}{\right)}
\newcommand{\bint}{\left[}
\newcommand{\eint}{\right]}
\newcommand{\bnorme}{\left\|}
\newcommand{\enorme}{\right\|}
\newcommand{\babsolu}{\left|}
\newcommand{\eabsolu}{\right|}

\newcommand{\be}{\begin{equation}}
\newcommand{\ee}{\end{equation}}
\newcommand{\ba}{\begin{array}}
\newcommand{\ea}{\end{array}}
\newcommand{\bdes}{\begin{description}}
\newcommand{\edes}{\end{description}}
\newcommand{\benu}{\begin{enumerate}}
\newcommand{\eenu}{\end{enumerate}}

\newcommand{\associe}{\mapsto}

\newtheorem{defi}{Definition}[section]
\newtheorem{lem}[defi]{Lemma}
\newtheorem{pro}[defi]{Proposition}
\newtheorem{thm}[defi]{Theorem}

\newtheorem{rmq}[defi]{Remark}
\newcommand{\dal}{\Delta^{\frac{\alpha}{2}}}

\newcommand{\pa}{{\frac{\alpha}{2}}}
\title{Blow up boundary solutions of some semilinear fractional equations in the unit ball}
\author{Mohamed Ben Chrouda$^a$ \and Mahmoud Ben Fredj$^b$}

\begin{document}
\maketitle
\begin{center}{\scriptsize{\emph{$^a$Dep. of Mathematics, High institut of computer sciences and mathematics, 5000, Monastir, Tunisia }}}\end{center}
\begin{center}{\scriptsize{\emph{$^b$Dep. of Mathematics, Faculty of Sciences of Monastir, 5019 Monastir, Tunisia}}}\end{center}
\begin{abstract} For $\gamma>0$, we are interested in blow up  solutions $u\in \cc{C}^+(B)$ of the fractional problem in the unit ball $B$
\begin{equation}\label{2nov}\bacc\ba{rcll}\dal u&=&u^\gamma&\ \text{in }B\\
u&=&0&\ \text{in }B^c.\ea\right.\end{equation}
We distinguish particularly two orders of singularity at the boundary:  solutions  exploding  at the same rate than $\delta^{1-\frac{\alpha}{2}}$ ($\delta$ denotes the Euclidean distance) and those  higher singular than $\delta^{1-\frac{\alpha}{2}}.$ As a consequence, it will be shown     that the classical Keller-Osserman condition can not be readopted in the fractional setting.
\end{abstract}
\section{Introduction}
Let $N\geq 3,$  $0<\alpha<2$ and let $(X_t)_{t\geq0}$ be the standard $\alpha$-stable process in $\rn.$ It is determined by its  characteristic function which takes the form
\[E^x\bp e^{i\xi(X_t-X_0)}\ep=e^{-t\babsolu \xi\eabsolu^\alpha}\quad;\quad \xi\in\rn,\]
where $E^x$ is the expectation with respect to the distribution $P^x$ of the process starting from $x\in\rn.$
It is a discontinuous Markov process and  give rise to equations with the fractional Laplacian $\dal$.

Nonlocal operators such as $\dal$ naturally arise in population dynamics, continuum mechanics,   game theory and some other fields, we quote for instance \cite{abdel,elgart,lieb, majda,vlahos,weitzner}, one can see also \cite{games}  for broader discussion.

In the classical setting, i.e. for $\alpha=2,$ one of the most commonly considered equations in the literature is the following:
\begin{equation}\label{ei}
\Delta u=\varphi(u)\quad\text{in }B,
\end{equation}
where $B$ is the unit ball of $\rn$ and $\varphi: [0,\infty[\to\rr$ is a some nonnegative nondecreasing function. Such equations appear naturally in several interesting contexts including  in the random systems of branching particles \cite{dynkin,gall}. Solutions of (\ref{ei}) verifying
\[\lim_{x\to\partial B}u(x)=\infty\]
are called large solutions or blow up (boundary) solutions.
 Merely for the sake of  completeness, we shall recall the pioneering work of \cite{keller} and \cite{osserman}. They proved independently that a   necessary and sufficient condition for the existence of a large solution to (\ref{ei}), where $\varphi$ is a positive nondecreasing function, is
\[\int_1^\infty\bp \int_0^s \varphi(t)\,dt\ep^{-\frac{1}{2}}\,ds<\infty,\]
that in the case of $\varphi(u)=u^\gamma,$ means $\gamma>1.$

The aim of the paper is to study \eqref{ei} substituting the classical Laplacian by one of its fractional powers. Our consideration is motivated by the natural question whether the classical results (particularly the Keller-Osserman condition) in this field may be extended to nonlocal operators.

The behavior of $\alpha$-harmonic functions contrasts, in some respects, with the one of the classical harmonic functions. Indeed, due to the jumping nature of the $\alpha$-stable process, roughly speaking, at the exit time one could end up anywhere outside the domain. Put differently, the process typically leaves domains by jumping to the interior of its complement while the continuous paths of Brownian motion leave domains by hitting the boundary. This spans the existence of positive harmonic functions on $B $ blowing up at the boundary. Such functions are called singular $\alpha$-harmonic functions and they are harmonic for the stable process $(X_t^B)_{t\geq0}$ killed on leaving $B$.
In the Brownian motion case, such functions do not exist due to the Fatou theorem and nontangentiel convergence of positive harmonic functions \cite{axler,6,wu78}.
In this sense, the Martin kernel $M_B^\alpha$
$$ \nu\to \int_{\partial B} \frac{(1-|x|^2)^\pa}{|x-y|^N}\nu(dy)\quad ;\quad x\in B
$$ provides a one-to-one correspondence between positive Random measures $\nu$ on  $\partial B$ and positive $\alpha$-harmonic functions on $B$  which supports the fact that singular $\alpha$-harmonic functions constitute, in some respects, the appropriate class of "harmonic functions". Nevertheless, the usual probabilistic interpretation of singular $\alpha$-harmonic functions on $B$ as solutions of Dirichlet problem is no longer true. The interested reader is referred to \cite[Teorem 3.18]{chenmartin}  where the authors provide some probabilistic interpretation of these functions. In particular,
\[ M_B^\alpha 1(x):=\int_{\partial B} \frac{(1-|x|^2)^\pa}{|x-y|^N}\sigma(dy)\quad ;\quad x\in B,\]
 defines a singular $\alpha$-harmonic on $B$ and it behaves like $\delta(x)^{\pa-1}$ on $B.$ Here,  $\sigma$ denotes the surface area measure on $\partial B$ and $\delta(x):=1-|x|$ is the Euclidean distance from $x$ to the boundary $\partial B.$ This entitles us  to study the following  appropriate reformulated semilinear Dirichlet problem associated to $\dal$ in $B$ taking in account the aspects raised  above.
 \begin{equation}\label{pp1}\bacc\ba{l}\dal u=u^\gamma\text{ in }B\\
 u=0\text{ in }B^c\\
 \displaystyle{\lim_{x\to z}\delta(x)^{1-\pa}u(x)=g(z)}\ea\right.
 \end{equation}
  where $\gamma>0$ and $g$ is a nonnegative continuous function on $\partial B.$
  It is worth noting that solutions of this problem explode on the boundary at the same  rate than $\delta^{\pa-1}$ while if we take $\alpha=2$ they don't. That is why solutions to the problem   (\ref{pp1}),  behaving like  the singular $\alpha$-harmonic function $M_B^\alpha 1$, in some respects, are not  blow up solutions in the strict sens of the word. We rather call them "moderate blow up solutions". In this spirit, we shall consider a second semilinear Dirichlet problem
   \begin{equation}\label{pp2}\bacc\ba{l}\dal u=u^\gamma\text{ in }B\\
 u=0\text{ in }B^c\\
 \displaystyle{\lim_{x\to z}\delta(x)^{1-\pa}u(x)=\infty}\ea\right.
 \end{equation}
 where $\gamma>0$ as before and we shall call its solutions by "blow up solutions". Note that they explode on the boundary even for $\alpha=2.$ In other words, such solutions,  are somehow the ones that   "large solution" (or "blow up solution") are  in the classical setting.

 Recently, a growing and renewed attention is paying to existence of boundary blow up solutions to equations involving  fractional powers of the Laplacian in bounded domains. Among the works, we quote \cite{maagli,nicola,chenfelmer,patriciofelmer} (the list is remotely incomplete).  
   In \cite{chenfelmer},  authors proved the existence of a real $\tau_0(\alpha)\in]-1,0[$ such that for
 \begin{equation}\label{gamma}
 1+\alpha<\gamma<1-\frac{\alpha}{\tau_0(\alpha)},
 \end{equation}
 the problem (\ref{2nov}) admits a solution $u$ verifying
 \[0<\liminf_{x\to\partial B}\delta(x)^{\frac{\alpha}{\gamma-1}}u(x)\leq \limsup_{x\to\partial B}\delta(x)^{\frac{\alpha}{\gamma-1}}u(x)<\infty.\]
 Let us emphasize from the very beginning  that (\ref{gamma}) turns out to be  nearly optimal and it will appear also in our work even though our treated problem is slightly different and the notion of solution is not the same. Indeed, as we will precise later, we deal with distributional solutions (those satisfying the equation when integrated  against a suitable set of test functions) and not with viscosity solutions (i.e. every smooth function touching from below or above the continuous solution is required to be a supersolution or a subsolution) as the case in \cite{patriciofelmer} or in \cite{chenfelmer}. We shall not elaborate further on this but we  just want to retain that for fractional equations, the notion of weak solution implies the one of viscosity solution \cite{servadei2014}.

 Let us now describe the main results of this paper more precisely and at the same time give the outline.  We first record some tools of potential theory pertaining to the killed process.  Next,  we prove some technical lemmas and in particular, we take a closer look at the potential $G_B^\alpha(\delta^{-\lambda}),$ $\lambda<1+\pa.$
  We shall afterwards  characterise  nonnegative solutions to problem (\ref{pp1}), in the case where $0<\gamma<\frac{2+\alpha}{2-\alpha},$ in terms of the Green operator $G_B^\alpha$ and the Martin kernel $M_B^\alpha.$ More precisely, we prove in this case that $u$ is a nonnegative solution of problem (\ref{pp1}) if and only if
\[u+G_B^\alpha u^\gamma=M_B^\alpha g.\]
 Then, we shall use the previous material to establish our first main result. More specifically, for every nontrivial boundary datum $g\in\cc{C}^+(\partial B),$ we prove that problem (\ref{pp1}) admits a  nonnegative solution  $u \in \cc{C}(B)$
if and only if $0<\gamma<\frac{2+\alpha}{2-\alpha}$ and that the solution is unique.
  The last section will be devoted to the problem (\ref{pp2}). Namely we divide the real half-line into four intervals. We prove that problem (\ref{pp2}) has a nonnegative solution $u\in\cc{C}(B)$ for $1+\alpha<\gamma<\frac{2+\alpha}{2-\alpha}$ and that it has no solution for $0<\gamma<1+\frac{\alpha}{2}.$ We have left open the ranges $1+\frac{\alpha}{2}\leq\gamma\leq 1+\alpha$ and $\gamma\geq \frac{2+\alpha}{2-\alpha}.$ It is worth noting that    the classical Keller-Osserman condition can not been fundamentally readopted in the framework of fractional Laplacian, else, taking $\alpha=2,$ we obtain that (\ref{pp2}) has no solution for $\gamma<2,$ which is not consistent with the Brownian motion case, because as we recalled above,  this problem admits a solution for every $\gamma\geq 1$ and hence for $1\leq \gamma<2.$ In other words, the classical Keller-Osserman condition is no longer the range where  large solutions exist or do not exist.

 Motivated by the considerable advances on extending potential-theoretic properties of Brownian motion to symmetric stable process \cite{bogdan1997,bogdan1999,bogdanby,bogdandistribution,chensong1997,chenmartin,chensong1998,kim2006}, there is often an interesting natural generalization of classical results. This work hopefully allows us to better understand how the nonlocal character and the jumping nature of stable processes influences the boundary behavior of solutions. We would like to point out that the approach which we follow is based essentially on some analytic tools from potential theory and it is completely different from that used in  \cite{chenfelmer}.  In this paper we do not use probabilistic techniques in an essential way, keeping them only for interpretation.
\section{Preliminaries}

 For every subset  $F$ of $\bb{R}^N$, let $\cc{B}(F)$ be the set of all Borel measurable functions on $F$ and
 let $\cc{C}(F)$ be the set of all continuous real-valued functions on $F.$ If $\cc{G}$ is a set of numerical
  functions then $\cc{G}^+$ (respectively $\cc{G}_b$) will denote the class of all functions in $\cc{G}$ which
  are nonnegative (respectively bounded). $\cc{C}^k(F)$ is the class of all functions that are $k$ times   continuously differentiable on $F.$ The uniform  norm will be denoted by $\left\|\cdot\right\|.$

  For every subset $A$ of $\rn$ we denote its complement by $A^c=\rn\sla A,$ its closure by $\overline{A}$ and its Euclidean boundary by $\partial A=\overline{A}\cap \overline{A^c}.$

  For two nonnegative functions $f$ and $g$, the notation $f\approx g$ means that there are positive constants $c_1$ and $c_2$ such that $c_1g(x)\leq f(x)\leq c_2g(x)$ in the common domain of definition for $f$ and $g.$

Let $\alpha\in ]0,2[$ and $N\geq3.$ The fractional Laplacian $\dal=-(-\Delta)^{\frac{\alpha}{2}}$ in $\rn$ is a nonlocal operator defined on the Schwartz class through the Fourier transform
\[\mathcal{F}\bp\bp-\Delta\ep^{\frac{\alpha}{2}}f\ep\bp\xi\ep=\bp 2\pi\babsolu \xi\eabsolu\ep^\alpha\cc{F}(f),\]
where $\cc{F}(f)(\xi)=\int_{\rn} e^{-2\pi iy\cdot \xi}f(y)dy$ up a constant. It does not act by pointwise differentiation but by a global integration with respect to a singular kernel: for every Borel function $u$ for which the integral exists,
\[\Delta^{\frac{\alpha}{2}}u(x)=c_{N,\alpha}\int_{\rn}\frac{u(y)-u(x)}{\babsolu y-x\eabsolu^{N+\alpha}}\,dy;\qquad x\in\rn,\]
 in such a way the nonlocal character  is emphasized.
Here the constant $c_{N,\alpha}$ is  depending only on $N$ and $\alpha:$ $c_{N,\alpha}=2^{\alpha}\pi^{-\frac{N}{2}}\babsolu\Gamma(\frac{-\alpha}{2})\eabsolu^{-1}\Gamma(\frac{N+\alpha}{2}).$ The natural space that we are going to use is a weighted $\cc{L}^1$-space:
\[\mathcal{L}_\alpha=\bacc u:\rn\to\rr\text{ Borel measurable };\ \int_{\rn}\frac{\babsolu u(y)\eabsolu}{\bp1+\babsolu y\eabsolu\ep^{N+\alpha}}dy<\infty\eacc.\]
In this space,  we can define $\dal $ as a distribution  by
\[<\dal u,\varphi>=\int_{\rn}u(y)\,\dal\varphi(y)\,dy;\quad \varphi\in \mathcal{C}_c^\infty(\rn),\]
where $ \mathcal{C}_c^\infty(\rn)$ is the set of all infinitely differentiable functions on $\rn$ with compact support.
In probabilistic terms, the fractional  operator $\dal$ is the infinitesimal generator of  $\alpha$-stable processes which constitute an important class of discontinuous Markov processes and which are widely used in several areas such as  mathematical finance,  probability, and physics. For more applications and a more complete account on the literature one can see the survey \cite{simulation}.

We denote by $(\Omega,X_t,P^x)$ the standard rotation  invariant  $\alpha$-stable $\rn$-valued Lévy process (homogeneous with independent increment and right continuous sample paths).
As usual, $E^x$ is the expectation with respect to the distribution $P^x$ of the process starting from $x\in\rn.$ For the sake of brevity, we will refer to this process by "symmetric $\alpha$-stable process". The limiting classical case $\alpha=2$ corresponds to the Brownian motion with Laplacian $\Delta=\sum_{i=1}^N\partial_i^2$ as generator.

Let $D$ be a bounded open set and let $X^D$ be the killed process defined by $X^D_t(\omega)=X_t(\omega)$ if $t< \tau_D(\omega)$ and set $X_t^D(\omega)=\partial$ if $t\geq \tau_D(\omega)$ where $\partial$ is a coffin state added to $\rn.$

Unlike the Brownian motion case, there are two different kinds of harmonicity with respect to symmetric $\alpha$-stable process: function which are harmonic in $D$ with respect to the process $X$ and functions which are harmonic in $D$ with respect to the killed process $X^D.$ The precise definitions are as follows.
\begin{defi}
Let $u\in \mathcal{L}_\alpha$ be a Borel measurable locally integrable function on $\rn.$
\benu
 \item We say that $u$ is $\alpha$-harmonic in $D$  if $u(x)=E^x[u(X_{\tau_U})], x\in U,$ for every bounded open set $U$ with closure $\overline{U}$ contained in $D.$ If, in addition, $u(x)=0,$ for every $x\in D^c,$ we say that $u$ is singular $\alpha$-harmonic.
 \item The function $u$ is said to be $\alpha$-superharmonic in $D$ if $f$ is lower semi continuous and $u(x)\geq E^x[u(X_{\tau_U})],$\  $x\in U,$ for every bounded open set $U$ with closure $\overline{U}$ contained in $D.$
 \item The function $u$ is called  $\alpha$-harmonic  with respect to $X^D$ if $u(x)=E^x[u(X_{\tau_U}^D)],$\  $x\in U,$ for every bounded open set $U$ with closure $\overline{U}$ contained in $D.$
 \eenu
 \end{defi}
 \begin{rmq}\label{rmq2}
 \benu
 \item Note that a singular $\alpha$-harmonic function in $D$ is $\alpha$-harmonic with respect to $X^D.$ Conversely, for a $\alpha$-harmonic function with respect to $X^D,$ if we extend it to be zero off the domain $D$ then the resulting function is singular $\alpha$-harmonic function in $D.$

 \item As in the classical case ($\alpha=2$) the $\alpha$-harmonicity can be formulated in terms of the fractional power of the Laplacian. Indeed,  a function $u\in\mathcal{L}_\alpha$ is $\alpha$-harmonic in $D$ if and only if it is continuous in $D$ and $\dal u=0$ in $D$ in the distributional sense (see e.g., \cite[Theorem 3.9]{bogdanby}).
 \item The following minimum principle for $\alpha$-superharmonic functions holds \cite{silvestre}. If $u$ is lower semi continuous on $\overline{D},$ $\alpha$-superharmonic in $D$ and $u\geq 0$ on $D^c$ then $u\geq 0$ in $D.$
 \eenu
 \end{rmq}

A point $x\in\partial D$ is called regular for the set $D$ if $P^x(\tau_D=0)=1.$ The (open) bounded domain $D$ is called regular if all $x\in \partial D$ are regular for $D$ (for instance,  $C^{1,1}$-domains and domains satisfying the exterior cone condition are regular). In this case each function $f,$ say in  $ \cc{C}_b(D^c),$ admits an  extension $H_D^\alpha f$ on $\rn$ such that $H_D^\alpha f$ is   $\alpha$-harmonic ~in ~$D$\cite{Landkof}. In other words,
  the function $h=H_D^\alpha f$ is the unique solution to the fractional Dirichlet problem
\[
\left\{
\begin{array}{rrrl}
\dal h&=&0& \mbox{in}\  D,\\
h&=&f&\textrm{in}\  D^c.
\end{array}
\right.
\]
 Note that $f$ must be defined on all $D^c$ because the data have to take into account the nonlocal character of the operator. For every $x\in D,$ the $\alpha$-harmonic measure relative to $x$ and $D,$
which will be denoted by $H_D^\alpha(x,\cdot),$
 is defined to be the positive Radon measure on $ D^c$ given by the mapping $f\mapsto H_D^\alpha f(x).$

 It is worthwhile to remark that the $\alpha$-harmonic measure for a  regular domain $D$ is typically defined (as a function) in the whole $\rn$ and is supported (as a measure) in the complementary of $D$ contrary to the classical harmonic measure which is defined  in $\overline{D}$ and supported on the boundary of $D.$ This can be better interpreted from a probabilistic point of view. It exhibits the fact that the sample paths of the $\alpha$-stable process fails to be continuous and hence at the exit time they may hit $D^c$ at points of $\overline{D}^c$ and not necessarily at points of $\partial D$ as the case of the Brownian motion.

 It is proved in \cite{bogdan1997} that for $D$ say Lipschitz, $H_D^\alpha(x,\cdot)$ is concentrated on  $\overline{D}^c$ and is absolutely continuous with respect to the Lebesgue measure on $D^c.$ Furthermore, the corresponding density function $K_D(x,y),$ $x\in D,$ $y\in D^c,$ is continuous in $(x,y)\in D\times$  ~$\overline{D}^c.$ In this situation, the solution of the Dirichlet problem can be expressed in term of the Poisson kernel $K_D$ as follows \cite{chensong1997}
 \[ H_D^\alpha f(x)=E^x\bint f(X_{\tau_D})\eint=\int_{D^c}K_D(x,y)\,f(y)\,dy\quad;\quad x\in D.\]
 The density $K_{B_r}^\alpha$ of the $\alpha$-harmonic measure for a ball $B_r$ of center $0$ and radius $r$ is given by an explicit formula
 \begin{equation}\label{poisson}
K_{B_r}(x,y)= c_\alpha^N \left(\frac{r^2-|x|^2}{|y|^2-r^2}\right)^{\frac{\alpha}{2}}\frac{1}{|x-y|^N},\;\; \textrm{ for }\, |x|<r\;\textrm{ and }\, |y|>r,
\end{equation}
where $c_\alpha^N:= \pi^{-1-N/2}\Gamma(N/2)\sin(\pi\alpha/2)$.
 The potential operator $G_D^\alpha$ for the process $(X_t^D)$ is defined  for every Borel measurable function~$f$ for which the following
   identity exists,  by
 \[G_D^\alpha f(x)=E^x\left[\int_{0}^{\infty}f(X_t^D)dt\right]=E^x\left[\int_{0}^{\tau_D}f(X_t)dt\right].\]
   If $G_D^\alpha f\not\equiv \infty$ then $G_D^\alpha f$ is $\alpha$-superharmonic. Also for every  $f\in \cc{B}_b(D)$,
 $G_D^\alpha f$ is a bounded continuous function on~$D$ satisfying $\lim_{x\rightarrow z}G_D^\alpha f(x)=~0$ for every  $z\in \partial D$ if we suppose further that $D$ is regular . The proof of these elementary properties follows  the line of the corresponding one for the Laplacian performed in \cite{12} or \cite{5}.
 On the other hand,  for every $f\in\mathcal{B}(\rn)$ such that $G_D^\alpha\babsolu f\eabsolu(x)<\infty$ for some $x\in\rn,$ we have
 \begin{equation}\label{operateurgreen}
 \dal G_D^\alpha f=-f
 \end{equation} in the distributional sense (  \cite[Lemma 5.3]{bogdandistribution}).
 The kernel $G_D^\alpha(\cdot,\cdot)$ of the operator $G_D^\alpha$ is called the Green function of $D$ and is defined by
$$
  G_D^\alpha f(x)=\int_DG_D^\alpha(x,y)f(y)dy\quad;\quad x\in D.
$$
 We shall record some well known facts of the Green function. $G_D^\alpha(x,y)$ is symmetric in $x$ and $y$ and $G_D^\alpha(\cdot,\cdot)$ is positive  and continuous in the extended sense as a mapping from $D\times D$ into $[0,\infty].$  Also $G_D^\alpha(x,y)=0$ if $x$ or $y$ belongs to $D^c.$   The Green function of the whole space $\rn,$ which is also called Riesz kernel, is given explicitly by
 \[G^\alpha_{\rn}(x,y)=\frac{C_{N,-\alpha}}{\babsolu x-y\eabsolu^{N-\alpha}}.\]   Also, the explicit formula for the Green function of a ball $B_r$ of center $0$ and radius $r$ is well known:
 \begin{equation}\label{green}
 G^\alpha_{B_r}(x,y)=\frac{\kappa_{N,\alpha}}{\babsolu x-y\eabsolu^{N-\alpha}}\int_0^{\frac{\bp r^2-\babsolu x\eabsolu^2\ep\bp r^2-\babsolu y\eabsolu^2\ep}{\babsolu x-y\eabsolu^2}}\frac{s^{\frac{\alpha}{2}-1}}{(1+s)^{\frac{N}{2}}}\,ds\quad;\quad x,y\in B_r,
 \end{equation}
 where $\kappa_{N,\alpha}=\Gamma(N/2)/\left(2^\alpha\pi^{N/2}[\Gamma(\alpha/2)]^2\right)$,
 and the following scaling property holds
 \begin{equation}\label{scaling}
 G_{B_r}^\alpha (x,y)=r^{\alpha-N}G_B^\alpha(\frac xr,\frac{y}{r})\quad;\quad x, y\in B_r.
 \end{equation}
  Let $D$ be a bounded $C^{1,1}$ domain in $\rn.$ We denote by  $\delta(x):=\inf_{z\in\partial D}|x-z|$   the Euclidean distance
from $x\in D$ to the boundary of~$D$. For convenience, we shall set $\delta(x)=0$ for $x\in D^c.$ The following sharp estimates on the Green function are established in   \cite{chensong1998}.
\begin{equation}\label{green1}
   G_D^\alpha(x,y)\leq c \min \bacc\frac{1}{|x-y|^{N-\alpha}},\frac{\delta(x)^{\frac{\alpha}{2}}\delta(y)^{\frac{\alpha}{2}}}{|x-y|^N}\eacc
  \end{equation}
 \begin{equation}\label{green5}
   G_D^\alpha(x,y)\leq c \frac{\delta_(y)^{\frac{\alpha}{2}}}{\babsolu x-y\eabsolu^{N-\pa}}
 \end{equation}
 \begin{equation}\label{mingreen}
 G_D^\alpha(x,y)\geq c \frac{\delta_(x)^{\frac{\alpha}{2}}\delta_(y)^{\frac{\alpha}{2}}}{\babsolu x-y\eabsolu^{N}}\;\; \textrm{ if }\;\; |x-y|>\max\{\frac{\delta(x)}{2},\frac{\delta(y)}{2}\}.
 \end{equation}
  Here and in the sequel, the letter $c$ with or without subscripts, signifies a positive constant which may change from one location to another (even in the same ligne) depending only  on $N$ and  $\alpha$ but eventually on another variable which will be specified in the context.
 \section{Some Technical lemmas}
 In this section, we shall give two preparatory technical lemmas but each of them turns out to be useful in itself. The first one provides an approximation for the fractional Laplacian value of some important subfamily of functions. The second one gives some useful properties for the potential $G_B^\alpha(\delta^{-\lambda}).$

 In the proof of the following lemma we will make use of the Gaussian hypergeometric function. We do not know one reference for all its basic properties but most of the material useful to our purposes can be found in \cite{erdely,table2,love}.
 \begin{lem}\label{fond1}
Let $-\frac{\alpha}{2}<\beta<1$ and consider the function $u$ defined on $\rn$ by
$$
u(x)=
  \frac{1}{(1-|x|^2)^{\beta}} \;\; \mbox{if }\; |x|< 1\quad\text{ and }\quad
  u(x)=0 \;\; \mbox{if }\; |x|\geq 1.
$$
Then
\begin{equation}\label{calcul}
 \dal u(x)\approx  \frac{\alpha+2\beta-2}{(1-|x|^2)^{\alpha+\beta}}\quad \text{in }B.
\end{equation}
\end{lem}
\begin{proof}
For every $0<t<1$, let $B_t$ denotes the ball of $\rn$ of center $0$ and radius $t$. For every $x\in B$, we denote $v(x):=\dal u(x)$. We first claim that,
for every $0<t<1$ and every $x\in B_t$,
$$
u(x)+G_{B_t}(v)(x)=H_{B_t}u(x),
$$
and in particular,
\be \label{uv0}
u(0)+G_{B_t}(v)(0)=H_{B_t}u(0).
\ee
In fact, using the integral expression of $\dal u$, it is not hard to see that $v$ is locally bounded on $B$. Thus, it follows from (\ref{operateurgreen}) that $\dal G_{B_t}^\alpha(v)=-v$ on $B_t$ and $\lim_{x\to z}G_{B_t}^\alpha(v)(x)=0$ for every $z\in \partial B_t$. This yields that the function $h$ defined on $\rn$ by $h(x)= u(x)+G_{B_t}(v)(x)$  is  $\alpha$-harmonic in $B_t$ and satisfies $h=u$ on $B_t^c$. Therefore, $h=H_{B_t}u$ on $B_t$ as desired.

By (\ref{poisson}) and using the spherical coordinates, for every $0<t<1$, we get
\begin{eqnarray}
H_{B_t}^\alpha u(0) &=& \frac{2\sin(\frac{\alpha\pi}{2})}{\pi}\int_t^1\frac{t^\alpha}{s(1-s^2)^\beta(s^2-t^2)^{\frac\alpha2}}ds\nonumber\\
            &=& C(\alpha,\beta)t^\alpha(1-t^2)^{1-\beta-\frac\alpha2}F(1, 1-\beta, 2-\beta-\frac\alpha2, 1-t^2)\label{tab1}\\
            &=& C(\alpha,\beta)(1-t^2)^{1-\beta-\frac\alpha2}F(1-\beta-\frac\alpha2, 1-\frac\alpha2, 2-\beta-\frac\alpha2, 1-t^2)\label{euler},
\end{eqnarray}
where $C(\alpha,\beta)$ is a positive constant.
Equalities (\ref{tab1}) and (\ref{euler}) are obtained, respectively, by formula (3.228) in \cite{table2} (or (33) page 250 in \cite{erdely}) and by the  Euler transformation of the Gaussian hypergeometric function $F(a, b, c, \cdot)$. On the other hand, it follows from (\ref{green}) that, for $|y|<1$,
\begin{eqnarray}
G_B(0,y) &=& \kappa_{N,\alpha}(1-|y|^2)^{\frac\alpha2}|y|^{-N}\int_0^1s^{\frac\alpha2-1}\left(1+(\frac{1-|y|^2}{|y|^2})s\right)^{-\frac N2}ds\nonumber\\
         &=& \frac{2\kappa_{N,\alpha}}{\alpha}|y|^{-N}(1-|y|^2)^{\frac\alpha2}F(\frac N2, \frac\alpha2, 1+\frac\alpha2, 1-\frac{1}{|y|^2})\nonumber\\
         &=& \kappa_{N,\alpha}|y|^{-N}\int_{|y|^2}^1(1-s)^{\frac\alpha2-1}F(\frac N2, \frac\alpha2, \frac\alpha2, 1-\frac{s}{|y|^2})ds \label{love}\\
         &=& \kappa_{N,\alpha}|y|^{-N}\int_{|y|^2}^1(1-s)^{\frac\alpha2-1} \left(\frac{s}{|y|^2}\right)^{-\frac N2} ds\nonumber\\
         &=& 2\kappa_{N,\alpha}\int_{|y|}^1s^{1-N}(1-s^2)^{\frac\alpha2-1}ds,\nonumber
\end{eqnarray}
where the equality~(\ref{love}) can be deduced from  formula (3.7) in \cite{love}. Then, using the scaling property (\ref{scaling}) and the spherical coordinates, we obtain
\begin{eqnarray*}
G_{B_t}v(0) &=& t^{\alpha-N}\int_{B_t}G_B(0,\frac{y}{t})v(y)dy\\
       &=& C(\alpha) t^{\alpha-N}\int_0^t\lambda^{N-1}v(\lambda)\int_{\frac{\lambda}{t}}^1s^{1-N}(1-s^2)^{\frac\alpha2-1}dsd\lambda\\
   &=&C(\alpha) \int_0^t\lambda^{N-1}v(\lambda)\int_\lambda^ts^{1-N}(t^2-s^2)^{\frac\alpha2-1}dsd\lambda\\
 &=& C(\alpha)\int_0^t(t^2-s^2)^{\frac\alpha2-1}s^{1-N}\int_0^s\lambda^{N-1}v(\lambda)d\lambda ds,
\end{eqnarray*}
where $C(\alpha)$ is a positive constant. Here, we identify $v(|y|)$ with $v(y)$ since $v$ is radially symmetric on $B$.
So, the equation
(\ref{uv0}) can be rewritten in the following equivalent form
\begin{equation}\label{fraceq1}
\int_0^t(t^2-s^2)^{\frac\alpha2-1}\psi(s)ds=\varphi(t),\;\;\forall\,0<t<1,
\end{equation}
where $\varphi(t):=H_{B_t}u(0)-u(0)$ and
\begin{equation}\label{p}
\psi(s):=C(\alpha)\,s^{1-N}\int_0^s\lambda^{N-1}v(\lambda)d\lambda.
\end{equation}
taking in account that $H_{B_t}u(0)\to u(0)$ as $t\to 0$,
 we can extend $\varphi$ continuously
by setting $\varphi(0)=0$. So, by a standard fractional calculus, we can express  $\psi$ in
terms of $\varphi$ in (\ref{fraceq1}) as follows
\begin{eqnarray}
\psi(t)&=& 2\frac{\sin(\alpha\pi/2)}{\pi}\left[\varphi(0)t^{1-\alpha}+t\,\int_0^t\frac{\varphi'(s)}{(t^2-s^2)^{\frac\alpha2}}\,ds\right]\nonumber\\
&=&2\frac{\sin(\alpha\pi/2)}{\pi}t\,\int_0^t\frac{\varphi'(s)}{(t^2-s^2)^{\frac\alpha2}}\,ds\label{pp}.
\end{eqnarray}
Next, using (\ref{euler}) and the fact that  $(z\partial_z+c).F(a, b, c+1, z)= c\,F(a, b, c, z)$, we easily obtain
\begin{equation}\label{ppp}
\varphi'(s)= (\alpha+2\beta-2)\,C(\alpha,\beta)\,\frac{s^{\alpha-1}}{(1-s^2)^{\beta+\frac\alpha2}}.
\end{equation}
Plugging (\ref{ppp}) into (\ref{pp}), we get
\begin{eqnarray}
\psi(t) &=& (\alpha+2\beta-2)C(\alpha,\beta)\frac{\sin(\alpha\pi/2)}{\pi}\, t \, \int_0^1\frac{s^{\frac\alpha2-1}}{(1-s)^{\frac\alpha2}(1-s\,t^2)^{\beta+\frac\alpha2}}ds\nonumber\\
        &=& (\alpha+2\beta-2)C(\alpha,\beta)\, t\, F(\beta+\frac\alpha2, \frac\alpha2, 1, t^2)\label{pppp}.
\end{eqnarray}
Multiplying (\ref{pppp}) by $t^{N-1}$, differentiating with respect to $t$, and combining this
together with (\ref{p}),
$$
C(\alpha)v(t) = (\alpha+2\beta-2)C(\alpha,\beta)\left[N\,F(\beta+\frac\alpha2, \frac\alpha2, 1, t^2)+ \frac{\alpha(\alpha+2\beta)}{2}\,t^2\,F(1+\beta+\frac\alpha2, 1+\frac\alpha2, 2, t^2)\right].
$$
By Euler transformation we get
$$
v(t) = \frac{(\alpha+2\beta-2)C(\alpha,\beta)}{C(\alpha)}\frac{1}{(1-t^2)^{\alpha+\beta}}\phi(t),
$$
with
$$
\phi(t)=N(1-t^2)F(1-\beta-\frac\alpha2, 1-\frac\alpha2, 1, t^2)+ \frac{\alpha(\alpha+2\beta)}{2}\,t^2\,F(1-\beta-\frac\alpha2, 1-\frac\alpha2, 2, t^2).
$$
Hence, the fact that the function $\phi$ is (strictly) positive and continuous on the closed interval $[0,1]$  complete the proof.
\end{proof}

Using the fact that $\min(s,t)\approx \frac{s\,t}{s+t}$ for $s,t>0,$ it follows from (\ref{green1}) that
\[G_B^\alpha(x,y)\approx \frac{1}{\babsolu x-y\eabsolu^{N-\alpha}}\frac{\delta(x)^{\frac{\alpha}{2}}\delta(y)^\pa}{(\babsolu x-y\eabsolu^2+\delta(x)\delta(y))^\pa}\cdot\]
Hence, for every $x,y\in B$ such that $\babsolu x-y\eabsolu \geq r$ we have
\begin{equation}\label{green3}
G_B^\alpha(x,y)\leq c \frac{\delta(x)^\pa\delta(y)^\pa}{r^N}\cdot
\end{equation}
\begin{lem}\label{continuity}
Let $\lambda\in \rr$ such that $\lambda<1+\frac{\alpha}{2}$ and consider the function $h$ defined on $B$ by
$$
h(x)=\int_B G_B^\alpha(x,y)\delta(y)^{-\lambda}dy.
$$
Then $h$ satisfies the following properties. \benu
\item[(a)]$h(x)<\infty$ for every $x\in B.$
\item[(b)] $h$ is continuous on $B$.
\item[(c)] $\lim_{x\to z}\delta(x)^{1-\pa}h(x)=0$ for every $z\in\partial B.$
\eenu
\end{lem}
\begin{proof} We prove all the assertions step by step.
\benu
\item[(a)] Let $x\in B$ and let $r>0$ be  small enough so that $B(x,r)\subset B.$ Using (\ref{green5}) and then splitting integral into two pieces we obtain
\begin{displaymath}
\begin{split}
\int_BG_B^\alpha(x,y)\delta(y)^{-\lambda}\,dy&\leq c\int_B \frac{\delta(y)^{\pa-\lambda}}{\babsolu x-y\eabsolu^{N-\pa}}\,dy\\
&=c \int_{B(x,r)}\frac{\delta(y)^{\pa-\lambda}}{\babsolu x-y\eabsolu^{N-\pa}}\,dy+ c \int_{B\sla B(x,r)}\frac{\delta(y)^{\pa-\lambda}}{\babsolu x-y\eabsolu^{N-\pa}}\,dy.
\end{split}
\end{displaymath}
For $y\in B(x,r)$ it is clear that $\delta(y)\geq \delta(x)-r$ and hence
\[\int_{B(x,r)}\frac{\delta(y)^{\pa-\lambda}}{\babsolu x-y\eabsolu^{N-\pa}}\,dy\leq  w(x) \int_{B(x,r)}\frac{dy}{\babsolu x-y\eabsolu^{N-\pa}}=c\,r^{\frac{\alpha}{2}}\,w(x)<\infty,\]
where $w(x)=\bp \delta(x)-r\ep^{\pa-\lambda}$ if $\lambda>\pa,$ and $w(x)=1$ otherwise. Now, for every $y\in B\sla B(x,r)$, $\babsolu x-y\eabsolu>r$ and then, noting that $\delta(y)=1-\babsolu y\eabsolu,$ we have
\begin{displaymath}
\begin{split}
\int_{B\sla B(x,r)}\frac{\delta(y)^{\pa-\lambda}}{\babsolu x-y\eabsolu^{N-\pa}}\,dy&\leq \frac{1}{r^{N-\pa}}\int_B \frac{1}{(1-\babsolu y\eabsolu)^{\lambda-\pa}}\,dy\\
&\leq \frac{c}{r^{N-\pa}}\int_0^1\frac{t^{N-1}}{(1-t)^{\lambda-\pa}},
\end{split}
\end{displaymath}
which is finite since $\lambda<1+\pa $ and hence the first assertion holds.
\item[(b)] Let $x_0\in B$ and $r>0$ so that $B(x_0,2r)\subset B.$ Let  $x$ such that $\babsolu x-x_0\eabsolu<r.$ We break up $h$ into two pieces
\[h(x)=G_B^\alpha\bp1_{B(x_0,2r)}\delta^{-\lambda}\ep(x)+G_B^\alpha\bp 1_{B\sla B(x_0,2r)}\delta^{-\lambda}\ep(x)=:h_1(x)+h_2(x).\]
The approach is as follows. We check  that $\babsolu h_1(x)-h_1(x_0)\eabsolu$ can be made arbitrarily small with a suitably chosen $r>0$ and that $\babsolu h_2(x)-h_2(x_0)\eabsolu\to 0$ as $\babsolu x-x_0\eabsolu\to 0$ for this choose of $r.$  For every  $y\in B(x_0,2r)$, $\delta(y)\geq \delta(x_0)-2r$. Thus, by (\ref{green1}), we have
\[G_B^\alpha(x,y)\delta^{-\lambda}(y)\leq c \bp \delta(x_0)-2r\ep^{-\lambda}\frac{1}{\babsolu x-y\eabsolu^{N-\alpha}}.\]
 Then, recalling that $m(B(x_0,2r))=c(2r)^N,$  we obtain
 \begin{displaymath}
\begin{split}
&\babsolu h_1(x)-h_1(x_0)\eabsolu=\int_{B(x_0,2r)}\babsolu G_B^\alpha(x,y)-G_B^\alpha(x_0,y)\eabsolu \delta^{-\lambda}(y)dy\\
&\leq c \bp \delta(x_0)-2r\ep^{-\lambda}\bp \int_{B(x_0,2r)}\frac{dy}{\babsolu x-y\eabsolu^{N-\alpha}}+\int_{B(x_0,2r)}\frac{dy}{\babsolu x_0-y\eabsolu^{N-\alpha}}\ep\\
&\leq  c \bp \delta(x_0)-2r\ep^{-\lambda}\bp \int_{B(x,r)}\frac{dy}{\babsolu x-y\eabsolu^{N-\alpha}}+\int_{B(x_0,2r)\sla B(x,r)}\frac{dy}{\babsolu x-y\eabsolu^{N-\alpha}}+(2r)^\alpha\ep\\
&\leq c \bp \delta(x_0)-2r\ep^{-\lambda} \bp r^\alpha+\frac{r^N}{r^{N-\alpha}}+r^\alpha\ep\\
&=c \bp \delta(x_0)-2r\ep^{-\lambda}r^\alpha.
\end{split}
\end{displaymath}
It is clear that the last term is arbitrarily small for a careful choose of $r>0.$

Now, having chosen $r,$ let us turn to the second addend $h_2$.   For every $y\in B\sla B(x_0,2r)$ and every $x\in B(x_0,r)$, we have $\babsolu x-y\eabsolu>r$ and $\babsolu x_0-y\eabsolu>r.$  Then it follows from (\ref{green3}) that
 \begin{displaymath}
 \begin{split}
 \babsolu G_B^\alpha(x,y)-G_B^\alpha(x_0,y)\eabsolu\delta(y)^{-\lambda}&\leq c\frac{\delta(x)^{\pa}+\delta(x_0)^{\pa}}{r^N}\delta(y)^{-\lambda+\frac{\alpha}{2}}\\
 &\leq \frac{c}{r^N} \delta(y)^{-\lambda+\frac{\alpha}{2}}.
 \end{split}
 \end{displaymath}
 Noting that $\delta(y)=1-\babsolu y\eabsolu$ and that $\lambda<1+\frac{\alpha}{2},$ we get
 \[\int_{B\sla B(x_0,2r)} \delta(y)^{-\lambda+\frac{\alpha}{2}}dy\leq\int_B \delta(y)^{-\lambda+\frac{\alpha}{2}}dy=c\int_0^1\frac{t^{N-1}}{(1-t)^{\lambda-\frac{\alpha}{2}}}dt<\infty.\]
 Hence, by dominated convergence theorem,
 \begin{displaymath}
 \begin{split}
 \lim_{\babsolu x-x_0\eabsolu\to 0}&\babsolu h_2(x)-h_2(x_0)\eabsolu\\
 &\leq\lim_{\babsolu x-x_0\eabsolu\to 0}\int_{B\sla B(x_0,2r)}|G_B^\alpha(x,y)-G_B^\alpha(x_0,y)|\delta(y)^{-\lambda}dy=0.
 \end{split}
\end{displaymath}
 We can now   conclude that $h$ is continuous at $x_0.$
 \item[(c)]  For every $-\pa<\beta<1-\pa$, applying $G_B^\alpha$ on both sides of (\ref{calcul}) and noting that $\delta(x)\approx 1-\babsolu x\eabsolu^2 $ we immediately obtain
 \[G_B^\alpha(\delta^{-(\alpha+\beta)})\approx \delta^{-\beta}\quad\text{in }B,\]
 that is, for every $\pa<\lambda'<1+\pa$,
 \[G_B^\alpha(\delta^{-\lambda'})\approx \delta^{\alpha-\lambda'}\quad \text{in }B.\]
 Now choose $\lambda'$ such that $\max(\pa,\lambda)<\lambda'<1+\pa$. Then $h\leq G_B^\alpha(\delta^{-\lambda'})$ on $B$ and consequently, for every $z\in \partial B$,
 \[\lim_{x\to z}\delta(x)^{1-\pa}h(x)\leq\lim_{x\to z}\delta(x)^{1-\pa}G_B^\alpha(\delta^{-\lambda'})=\lim_{x\to z}\delta(x)^{1+\pa-\lambda'}=0.\]
 \eenu
\end{proof}

 \section{Moderate blow up solutions}

 For $\gamma>0,$ we  consider the following fractional equation
 \begin{equation}\label{ep1}
 \dal u=u^\gamma.
 \end{equation}
 Before we carry on, we make precise the notion of solution that we use in this paper. By a solution of the Eq. (\ref{ep1}) in a open set $U\subset\rn,$ we shall mean every real-valued nonnegative function $u\in\cc{C}(U)\cap\mathcal{L}_\alpha$ such that
\[\int_{\rn} u(x)\,\dal \varphi(x)\,dx=\int_{\rn} u^\gamma(x)\,\varphi(x)\,dx\]
holds for every nonnegative function $\varphi$ belonging to the space $\mathcal{C}_c^\infty(U)$.
We first quote from \cite{chrouda} the following lemma
which states a straightforward and useful fact.
\begin{lem}\label{l1}  Let $U$ be an open set and let $u\in \cc{C}(U)\cap\mathcal{L}_\alpha$.
 The function $u$ is a solution of $\dal u=u^\gamma$ in  $U$ if and only if $u+G_D^\alpha\bp u^\gamma\ep=H_D^\alpha u$ holds for every regular open set $D\subset \overline{D} \subset U.$
\end{lem}
\begin{lem}\label{comparison}  Let $D$ be an open regular bounded domain in $\rn.$ Let
$u,v\in \cc{C}(\overline{D})\cap \cc{L}_\alpha$ such that
$$
\dal u\leq u^\gamma\qquad \mbox{and} \qquad \dal v\geq v^\gamma\qquad \mbox{in}\ D.
$$
If $u\geq v$ on $ D^c,$ then $u\geq v$ in $\rn.$
\end{lem}

 \begin{proof}  Define $w=u-v.$ Arguing by contradiction, wa assume    that the open set
$$
U=\left\{x\in D;w(x)<0\right\}
$$
 is not empty.  We get immediately  that  $\dal w\leq \bp u^\gamma-v^\gamma\ep\leq 0$
 in~$U$, this amounts to say that  $w$ is $\alpha$-superharmonic in~$U.$ Furthermore,  we have $w\geq 0$ on $D^c$. Therefore, $w\geq 0$ in $U$ by the minimum principle  for $\alpha$-superharmonic functions (Remark \ref{rmq2} (c))    and this yields a contradiction. Consequently, $U=\emptyset$ and hence  $u\geq v$ in $D.$
\end{proof}

It should be noticed that the comparison principle as stated in the above lemma become more or less a classical result and it is widely used in the literature ( \cite{felmerlarge,chrouda}...). Also, as we alluded to before, the minimum principle and comparison results require information on the solutions in the whole complement of the domain and not only  at the boundary, consistently with the nonlocal character of the operator $\dal$.

Let $D$ be an open bounded $C^{1,1}$ domain in $\rn.$ It was shown in \cite{chrouda} that, for every $f\in\cc{C}_b^+(D^c),$ there exists one and only one function  $u\in\cc{C}^+(\overline{D})$ solution of the following semilinear Dirichlet problem
\begin{equation}\label{dirichlet}
\bacc\ba{rcl}\dal u&=&u^\gamma \text{ in } D.\\
u&=&f\text{ on } D^c\ea\right.
\end{equation}
As a matter of fact, the continuity of $f$ in the whole complement of $D$ is not vital for the existence of a solution $u\in\cc{C}(\overline{D})$ but it guarantees rather the continuity of the solution in the whole space $\rn$ and this is not important for our purposes. In this context, we need only to assume that $f\in \cc{C}(\partial D)\cap \cc{B}^+_b(D^c)$ rather than $f\in\cc{C}_b^+(D^c)$. In the following proposition, we would like to prove the existence of a solution $u\in\cc{C}^+(\overline{D})$ to the same problem but dropping the boundedness of the boundary datum $f,$ thus extending the result in \cite{chrouda}.
\begin{pro}\label{thm81}  For every nonnegative function $f\in \cc{C}(\partial D)\cap \mathcal{L}_\alpha$, there exists one and only one function   $u\in\cc{C}^+(\overline{D})$  solution of the
problem (\ref{dirichlet}).
 Moreover, for every $x\in\rn,$
$$
u(x)+G_D^\alpha(u^\gamma)(x)=H_D^\alpha f(x).
$$
\end{pro}
\begin{proof}   For every $n\geq1$, put $f_n=\min(f,n)$. Obviously, $f_n\in \cc{C}(\partial D)\cap \cc{B}^+_b(D^c)$. Then, as mentioned above, there exists a  nonnegative function $u_n\in \cc{C}^+(\overline{D})$ solution to the problem (\ref{dirichlet}) with boundary datum $f_n$ instead of $f.$ Furthermore, by Lemma \ref{l1},
\begin{equation}\label{a1}
u_n(x)+G^\alpha_D(u_n^\gamma)(x)=H_D^\alpha f_n(x)\quad;\quad x\in\rn.
\end{equation}
On account of  the comparison principle,  $(u_n)$ is nondecreasing. If we put $u:=\sup _n u_n$ then obviously, $u\geq0$. On the other hand, letting $n$ tend to $\infty$ in (\ref{a1}), we obtain
\begin{equation}\label{relation}
u(x)+G_D^\alpha(u^\gamma)(x)=H_D^\alpha f(x)\quad ;\quad x\in\rn.
\end{equation}
 We  readily observe that $u$ is bounded on $D$ by $\sup_{x\in D}H_D^\alpha f(x)$ and thereby $u^\gamma$ is also bounded on $D$. This yields that $G_D^\alpha(u^\gamma)\in\cc{C}_0(D)$, the subclass of $\cc{C}(\overline{D})$ consisting of  functions which vanish continuously on $\partial D.$  Then, from (\ref{relation}), we immediately  deduce that   $u\in \cc{C}(\overline{D})$ and that $u=f$ on $D^c$. Using (\ref{operateurgreen}), we also deduce from (\ref{relation})  that $\dal u=u^\gamma$ in $D$.  Hence,   $u$ is a solution to the problem (\ref{dirichlet}) as desired. The uniqueness of the solution follows immediately from Lemma~\ref{comparison}.
\end{proof}

Now we wish to spend a few words on the Martin kernel. Let $D$ be an open bounded $C^{1,1}$ domain in $\rn$ and let $x_0\in D$ be a reference point. The Martin kernel of the killed symmetric stable process is defined by
\[M_D^\alpha(x,z)=\lim_{y\to z}\frac{G_D^\alpha(x,y)}{G_D^\alpha(x_0,y)},\quad\text{for }x\in D\ \text{and } z\in\partial D,\]
and it has properties similar to those of the Martin kernel with respect to the killed Brownian motion \cite{bogdan1999}: the mapping $(x,z)\associe M_D^\alpha(x,z)$ is continuous on $D\times\partial D.$ Moreover, for every $z\in\partial D,$ $M_D^\alpha(\cdot,z)$ is a positive singular $\alpha$-harmonic in $D$ and for every $z,w\in\partial D$ such that $z\neq w$ we have $\lim_{x\to w}M_D^\alpha(x,z)=0.$ It is proved also that the formula
\[h(x)=\int_{\partial D}M_D^\alpha(x,z)\,d\nu(z)\]
realizes a one-to-one correspondence between nonnegative singular $\alpha$-harmonic functions on $D$ and positive Radon measures on $\partial D.$
 The Martin kernel of the unit ball $B$ is given explicitly by
$$
M_B^\alpha(x,y)=\frac{(1-|x|^2)^{\alpha/2}}{|x-y|^N}\quad;\quad x\in B,\, z\in\partial B.
$$
For every nonnegative continuous function $g$ on $\partial B$, we denote
$$
M_B^\alpha g(x):= \left\{
            \begin{array}{ll}
\displaystyle{\int_{\partial B} M_B^\alpha(x,y)g(y) d\sigma(y)}& \;\;\textrm{ if }\, x\in B\\
0 &  \;\;\textrm{ if }\, x\notin B,
 \end{array}
 \right.
$$
where $\sigma$ is the surface area measure on $\partial B$. An elementary calculation leads to the following estimates.
\begin{equation}\label{kim2}
M^\alpha_B1(x)\approx \delta(x)^{\pa-1}
\end{equation}
 Furthermore, it follows from \cite[Theorem 3.18]{kim2006} that $M_B^\alpha g$ is the unique positive solution $h$ of the following boundary value problem:

\begin{equation}\label{kimproblem}
\left\{
\begin{array}{ll}
\dal h=0 \;\textrm{ in }\, B\\
h=0 \;\textrm{ on }\, B^c\\
\displaystyle{\lim_{x\to z\in\partial B}\delta(x)^{1-\pa} h(x)=g(z)}.
 \end{array}
 \right.
\end{equation}

For $g\in \cc{C}^+(\partial B)$, we shall investigate the existence and the uniqueness of a solution to the  following problem:
\begin{equation}\label{p1}
\bacc\ba{l}\dal u=u^\gamma\text{ in } B\\
u=0 \;\textrm{ on }\, B^c\\
\displaystyle{\lim_{x\to z\in\partial B}\delta(x)^{1-\frac{\alpha}{2}}u(x)=g(z).}\ea\right.
\end{equation}
Clearly, solutions of this problem explode at the boundary. The explosion is driven by the function $g$ and the explosion rate has to be controlled by $\delta(x)^{1-\pa}.$

To put on readiness, we shall  give a useful characterization of solutions to problem (\ref{p1}) in terms of the Green operator and the Martin kernel. This characterization will be the key ingredient to prove the uniqueness as it is stated in the  forthcoming main theorem.
\begin{pro}\label{unicity} Let $0<\gamma<\frac{2+\alpha}{2-\alpha}.$ A function $u\in \cc{B}^+(\rn)$  is a  solution to the problem (\ref{p1}) if and only if $u+G_B^\alpha u^\gamma=M_B^\alpha g$ on $\rn$.
\end{pro}
\begin{proof} Let $u\in \cc{B}^+(\rn)$. Suppose first that  $u$  is a  solution to the problem (\ref{p1})  and  consider $w:=u+G_B^\alpha  (u^\gamma)$.
Since $M_B^\alpha g$ is the unique solution of the problem (\ref{kimproblem}), it will be sufficient to show that $w$ is also a solution of this problem. Taking in account that, for every $z\in \partial B$, $\lim_{x\to z}\delta^{1-\frac\alpha2}(x)u(x)=g(z)$, we can extend $\delta^{1-\frac\alpha2}u$ continuously on $\overline{B}$  and consequently there exists some constant $C>0$ such that $u\leq C\delta^{\frac\alpha2-1}$ on $B$.
 So, for every $x\in B$,
$$
G_B^\alpha(u^\gamma)(x)\leq C^\gamma G_B^\alpha\left(\delta^{\gamma(\pa-1)}\right)(x).
$$
Then, since $\gamma(1-\pa)<1+\pa$, it follows from assertion (a) in Lemma \ref{continuity} that $G_B^\alpha(u^\gamma)(x)<\infty$ which yields, using (\ref{operateurgreen}), that $\dal G_B^\alpha(u^\gamma)=-u^\gamma$ in $B$. Thus $\dal w=0$ in $B$.
Also, we deduce using assertion (c) in Lemma \ref{continuity} that  $\lim_{x\to z} \delta(x)^{1-\pa}G_B^\alpha(u^\gamma)(x)=0$ and  hence
$$
\lim_{x\to z} \delta(x)^{1-\pa}w(x)=\lim_{x\to z} \delta(x)^{1-\pa}u(x)=g(z).
$$

Conversely, suppose that $u$ satisfies $u+G_B^\alpha (u^\gamma)=M_B^\alpha g$ on $\rn$. Obviously, $\dal u=u^\gamma$ in $B$, $u=0$ on $B^c$ and  $ u\leq \bnorme g\enorme M_B^\alpha 1\leq c \bnorme g\enorme\delta^{\pa-1}.$ Then, in view of the  assertion (c) in Lemma \ref{continuity}, we obtain
$$
\lim_{x\to z}\delta(x)^{1-\pa}G_B^\alpha(u^\gamma)(x)\leq (c\|g\|)^\gamma\lim_{x\to z} \delta(x)^{1-\pa}G_B^\alpha\bp \delta^{\gamma(\frac{\alpha}{2}-1)}\ep(x)=0.
$$
This entails that
\[\lim_{x\to z}\delta(x)^{1-\pa}u(x)=\lim_{x\to z}\delta(x)^{1-\pa} M_B^\alpha g(x)=g(z).\]
Now it remains to prove that $u$ is continuous in $B.$
Define $v(x)=\bnorme g\enorme\bp\delta(x)\ep^{\gamma(\frac{\alpha}{2}-1)}-u^\gamma(x)$ for $x\in B.$ Obviously, $v$ is nonnegative and
\[G_B^\alpha(u^\gamma)+G_B^\alpha v=\bnorme g\enorme G_B^\alpha\bp \delta^{\gamma(\frac{\alpha}{2}-1)}\ep\;\;\textrm{ on }\; B.\]
By the assertion (b) in Lemma \ref{continuity}, $G_B^\alpha\bp \delta^{\gamma(\frac{\alpha}{2}-1)}\ep\in \cc{C}(B)$. Then, seeing  that $G_B^\alpha(u^\gamma)$ and $G_B^\alpha v$ are lower semi-continuous on $B$, we immediately conclude that $G_B^\alpha(u^\gamma)\in \cc{C}(B)$
 and hence the continuity of $u$ on $B$ holds since $u+G_B^\alpha (u^\gamma)=M_B^\alpha g$ and $M_B^\alpha g\in \cc{C}(B)$. Consequently, $u$ is a solution to the problem (\ref{p1}) and the proof is finished.
\end{proof}
\begin{thm} Let  $\gamma>0$ and let $g\in\cc{C}^+(\partial D)$ not identically vanishing.
\benu  \item If $\;0<\gamma<\frac{2+\alpha}{2-\alpha},$ then problem (\ref{p1}) admits one and only one nonnegative solution.
\item If $\;\gamma\geq\frac{2+\alpha}{2-\alpha}$, then problem (\ref{p1}) has no nonnegative solution.
\eenu
\end{thm}
\begin{proof}\benu \item
Let $0<\gamma<\frac{2+\alpha}{2-\alpha}$. For every $n\geq 2$, we denote by  $B_n$ the ball of $\rn$ of center $0$ and radius $1-\frac{1}{n}$. In virtue of proposition \ref{thm81}, invoking that $M_B^\alpha g\in\cc{C}(\partial B_n)\cap \cc{L}_\alpha,$ for every $n\geq 1,$ there exists an unique    $u_n\in\cc{C}^+(\overline{B}_n)$ solution of the problem
$$
\bacc\ba{rcll}\dal u_n&=&u_n^\gamma&\text{ in } B_n\\
u_n&=&M_B^\alpha g&\text{ on }B_n^c.
\ea\right.
$$
Moreover,
\begin{equation}\label{a2}
u_n(x)+G_{B_n}^\alpha(u_n^\gamma)(x)=M_B^\alpha g(x)\quad;\quad x\in\rn.
\end{equation}
For every $n\geq 2$, we readily observe  that  $u_n\leq M_B^\alpha g$ which yields using Lemma \ref{comparison} that $u_{n+1}\leq u_n$. Denote $u=\inf_n u_n.$
On account of  (\ref{kim2}), for every $x, y \in \rn$  we have
\[G_{B_n}^\alpha(x,y)\,u_n^\gamma(y)\leq G_B^\alpha(x,y)(M_B^\alpha g)^\gamma(y)
\leq c\bnorme g\enorme^\gamma G_B^\alpha(x,y) \delta(y)^{\gamma(\frac{\alpha}{2}-1)}.
\]
Bearing in mind the hypothesis $0<\gamma<\frac{2+\alpha}{2-\alpha}$, we deduce from assertion (a) in Lemma \ref{continuity} that
\[\int_B G_B^\alpha(x,y)\delta(y)^{\gamma(\pa-1)}dy<\infty.\]
Thus, by the dominated convergence theorem,
$$
\lim_{n\to\infty}G_{B_n}^\alpha(u_n^\gamma)(x)=G_{B}^\alpha(u^\gamma)(x).
$$
Consequently, by letting $n$ tend to $\infty$ in (\ref{a2}), we obtain
$$
u(x)+G_B^\alpha(u^\gamma)(x)=M_B^\alpha g(x)\quad x\in\rn,
$$
and hence $u$ is a nonnegative solution to problem (\ref{p1}) in virtue of  Proposition \ref{unicity}.

Let us now prove the uniqueness. Let $u$ and $v$ be tow solutions of the problem (\ref{p1}). We denote $w:=u-v$ and $h$ the nonnegative function defined on $B$ by
\[h(x):=\bacc \ba{ll}\frac{u^\gamma(x)-v^\gamma(x)}{u(x)-v(x)}&\text{ if }u(x)\neq v(x)\\
0&\text{ if } u(x)=v(x)\ea\right.\]
Since, by proposition \ref{unicity}, $u+G_B^\alpha (u^\gamma)=M_B^\alpha g=v+G_B^\alpha (v^\gamma)$, we immediately deduce that $w+G_B^\alpha(wh)=0$ which is equivalent to
 $w+G_B^\alpha(hw^+)=G_B^\alpha(hw^-)$ since $G_B^\alpha(h \babsolu w\eabsolu)<\infty.$ Here and according to the notational custom, $w^+=\max(w,0)$ and $w^-=\max(-w,0)$.
In particular, $G_B^\alpha(hw^+)\leq G_B^\alpha(hw^-)$ on $\bacc w>0\eacc=\bacc w^+>0\eacc.$ Then, by the dominated principle \cite{hansen}, we must have $G_B^\alpha(hw^+)\leq G_B^\alpha(hw^-)$ on $B$ and hence $G_B^\alpha(h\,w)\leq 0=w+G_B^\alpha(h\,w),$ which yields that $w\geq0.$ Consequently, $0\leq w\leq w+G_B^\alpha(hw)=0$ and so $w=0$ as desired.
\item Let $\;\gamma\geq\frac{2+\alpha}{2-\alpha}$. Arguing by contradiction, we suppose that the problem (\ref{p1}) admits a nonnegative solution $u$.
Seeing that the function $\delta^{1-\pa}u$ is a continuous extension of $g$ on $\overline{B}$,
 we can find $c>0$ so that
$$
 \delta(x)^{1-\pa}u(x)\leq c\quad\text{ for every }x\in B.
$$
 Thus $H_{B_t}^\alpha u(0)\leq c H_{B_t}^\alpha (\delta^{\pa-1})(0)$ for every $0<t<1$ where $B_t,$ as before, is the ball of center $0$ and radius $t$.
 Since $\delta(y)\approx 1-\babsolu y\eabsolu^2$, it follows from (\ref{euler}) that the function $t\to H_{B_t}^\alpha (\delta^{\pa-1})(0)$ is  bounded on $[0,1]$. But, in virtue of Lemma \ref{l1}, for every $0<t<1$ we have
\[u(0)+G_{B_t}^\alpha(u^\gamma)(0)=H_{B_t}^\alpha u(0).\]
Thus, letting $t$ tend to $1$, we obtain
 \begin{eqnarray}
 u(0)+G_B^\alpha(u^\gamma)(0)&=&u(0)+\lim_{t\uparrow 1}G_{B_t}^\alpha(u^\gamma)(0)=\lim_{t\uparrow 1} H_{B_t}^\alpha u(0)<\infty.\label{a30}
 \end{eqnarray}
 On the other hand, let $z\in\partial B$ such that $g(z)>0$. Since, by hypothesis, $\lim_{x\to z}\delta(x)^{1-\pa}u(x)=g(z)$, there exist  $r,\eta>0$ small enough  such that
 \[\delta(x)^{1-\pa}u(x)\geq \eta\quad\text{ for every }x\in B\cap B(z,r).\]
 Then, using (\ref{mingreen}), we get
 \begin{eqnarray*}
 G_B^\alpha(u^\gamma)(0)&=&\int_BG_B^\alpha(0,y)u^\gamma(y)\,dy\\
 &\geq&\int_{\frac 12<\babsolu y\eabsolu<1}G_B^\alpha(0,y)u^\gamma(y)\,dy\\
 &\geq&c\,\eta^\gamma \int_{B\cap B(z,r)}\frac{\delta(y)^{\gamma(\pa-1)+\pa}}{\babsolu y\eabsolu^N}dy\\
 &=&\infty
 \end{eqnarray*}
  contradicting  (\ref{a30}). Hence, problem (\ref{p1}) has no solution.
 \eenu
\end{proof}
\section{Blow up solutions}

For $\gamma>0$, we shall investigate the existence of a nonnegative solution of the following  problem:
  \begin{equation}\label{pp22}\bacc\ba{l}\dal u=u^\gamma\text{ in }B\\
 u=0\text{ in }B^c\\
 \displaystyle{\lim_{x\to z\in\partial B}\delta(x)^{1-\pa}u(x)=\infty.}
 \ea\right.
 \end{equation}
 More precisely, we divide the positive real half line into different intervals, where (\ref{pp22}) admits no solutions and  at least one solution. In the later case, we give a  simple equivalent to the solution. Our proof does not require heavy computations. We would point out that solutions to problem (\ref{pp22}) are relatively with coarse singularity at the boundary and hence, they are, in some respects, the fractional counterpart of large solutions in the classical case $\alpha=2.$
\begin{thm}\label{th2}\benu
\item If $\,1+\alpha<\gamma<\frac{2+\alpha}{2-\alpha}$, then the problem (\ref{pp22}) admits a nonnegative solution $u$.
Furthermore,
\begin{equation}\label{encad}
 u\approx \delta^{\frac{\alpha}{1-\gamma}}\quad\text{ in }B.
\end{equation}
\item If $\,0<\gamma<1+\frac\alpha2$, then problem (\ref{pp22}) has no nonnegative solution.
\eenu
\end{thm}
\begin{proof}
\benu\item
Let $1+\alpha<\gamma<\frac{2+\alpha}{2-\alpha}$ and put $\beta := \frac{\alpha}{\gamma-1}$. It is obvious that $\beta$ satisfies $1-\pa<\beta<1$.
Let $\vartheta$ be the function defined on $\rn$ as follows.
$$
\vartheta(x)=
  \frac{1}{(1-|x|^2)^{\beta}} \;\; \mbox{if }\; |x|< 1\quad\text{ and }\quad
  \vartheta(x)=0 \;\; \mbox{if }\; |x|\geq 1.
$$
Then, by Lemma \ref{fond1}, there exists some constant $C>1$ such that, for every $x\in B$,
\begin{equation}\label{a20}
\frac{1}{C}\frac{1}{(1-\babsolu x\eabsolu^2)^{\alpha+\beta}}\leq \dal\vartheta(x)\leq C \frac{1}{(1-\babsolu x\eabsolu^2)^{\alpha+\beta}}.
\end{equation}
Let
 $K > 1$ be so that $C= K^{\gamma-1}$ and consider the functions $v:=K\,\vartheta$ and $w:=\frac 1K\,\vartheta$.
  Clearly,  $w\leq v$ and
$$
\dal v\leq v^\gamma \;\; \textrm{ and }\;\; \dal w \geq w^\gamma \;\textrm{ in }\,B.
$$
For every $n\geq 2$, we denote by  $B_n$ the ball of $\rn$ of center $0$ and radius $1-\frac{1}{n}$. Since $v\in\cc{C}(\partial B_n)\cap \cc{L}_\alpha$,  by proposition \ref{thm81} there exists a unique   $u_n\in\cc{C}^+(\overline{B}_n)$ solution of the problem
$$
\left\{
\begin{array}{rcll}
\dal u_n&=&u_n^\gamma &\textrm{ in }\, B_n\\
u_n&=&v &\textrm{ on }\, B_n^c.
 \end{array}
 \right.
$$
Moreover, using Lemma \ref{comparison}, it is easy to see that
$$
w\leq u_{n+1}\leq u_n \leq v.
$$
Put $u:=\inf_nu_n$, then we immediately deduce from these inequalities that    $u\in \cc{L}_\alpha$, $u=0$ on $B^c$, $u$ satisfies (\ref{encad}) and that $\lim_{x\to z\in\partial B}\delta(x)^{1-\pa}u(x)=\infty$ for every $z\in \partial B$. So, it remains to prove that $u\in C(B)$ and $\dal u= u^\gamma$ in $B.$  Let $D$ be a regular open set such that $\overline{D}\subset B$ and let $n_0$ be the smallest integer $n\geq2$ such that $\overline{D}\subset B_n$. Then for every $n\geq n_0$ and every $x\in \rn$,
\begin{equation}\label{dern}
u_n(x)+G_D^\alpha(u_n^\gamma)(x)= H_D^\alpha u_n(x)
\end{equation}
since $\dal u_n = u_n^\gamma$ in $D$ and $u_n\in \cc{C}(D)\cap \cc{L}_\alpha$. Whence, by letting $n$ tend to $\infty$ in (\ref{dern}), we obtain
 $$
u(x)+G_D^\alpha(u^\gamma)(x)=H_D^\alpha u(x)
$$
 which  entails that $u\in \cc{C}(D)$ because $H_D^\alpha u$ and $G_D^\alpha (u^\gamma)$ are in $\cc{C}(D)$. Since $D\subset B$ is arbitrary, we conclude that $u\in \cc{C}(B)$ and, by Lemma \ref{l1}, that  $\dal u= u^\gamma$ in $B$. This ends up the first part of the proof.
\item Let $\,0<\gamma<1+\frac\alpha2$ and assume that the problem (\ref{pp22}) admits a nonnegative solution $u$. We first claim that
\begin{equation}\label{ginfty}
G_B^\alpha(u^\gamma)(x)=\infty \;\textrm{ for every }\, x\in B.
\end{equation}
In fact, assume that $G_B^\alpha(u^\gamma)(x_0)<\infty$ for some $x_0\in B$ and consider the function $h$ defined on $\rn$ by
$$
h(x):= u(x)+G_B^\alpha(u^\gamma)(x)\; \textrm{ if }\,x\in B\,\textrm{ and }\;h(x)=0 \;\textrm{ if }\, x\notin B.
$$
Obviously, $h$ is singular $\alpha$-harmonic in $B$ since, by (\ref{operateurgreen}), $\dal G_B^\alpha(u^\gamma)= -u^\gamma$ in $B$. Then, it follows from \cite{kim2006} that $\lim_{x\to z}\delta(x)^{1-\pa}h(x)<\infty$ for some $z\in \partial B$. But this leads to a contradiction because $\lim_{x\to z}\delta(x)^{1-\pa}h(x)\geq \lim_{x\to z}\delta(x)^{1-\pa}u(x)$ and $\lim_{x\to z}\delta(x)^{1-\pa}u(x)=\infty$ by hypothesis. So, the claim is checked. On the other hand,  $u$ is Lebesgue integrable since it is in $\cc{L}_\alpha$ and then the function $\delta u$ must be bounded on $\overline{B}$ which yields the existence of some constant $c>0$ such that $u\leq c\,\delta^{-1}$ on $B$ and hence
$$
G_B^\alpha(u^\gamma)\leq c^\gamma\,G_B^\alpha(\delta^{-\gamma})\,\textrm{ on }\, B.
$$
Now, minding the hypothesis $\,0<\gamma<1+\frac\alpha2$, we deduce from assertion (a) in Lemma \ref{continuity} that $G_B^\alpha(\delta^{-\gamma})(x)<\infty$ for every $x\in B$ and consequently $G_B^\alpha(u^\gamma)(x)<\infty$ contradicting the claim (\ref{ginfty}). Hence, under the condition $\,0<\gamma<1+\frac\alpha2$, the problem (\ref{pp22}) has no nonnegative solution.
\eenu
\end{proof}
The restriction "$\gamma>\alpha+1$" seems to be technical but nearly optimal for the existence of solutions to problem (\ref{pp22}). We remember that this exponent "$\alpha+1$" appeared in \cite{chenfelmer} too. On the other hand, taking $\alpha=2,$ in virtue of   Theorem \ref{th2}, there exists no solution to problem (\ref{pp22}) for $0<\gamma<2$ and this is not consistent with the classical Keller-Osserman condition. In fact, it is well known that, in the old Laplacian case and taking in account the Keller-Osserman condition,  (large) solutions exist for every $\gamma>1$ and hence for $1<\gamma<2.$   Put differently, the classical Keller-Osserman condition, in the fractional setting, constitute  no longer the appropriate range where large solutions exists or do not exist as  in the Laplacian framework.

\end{document}